\documentclass[10pt,a4paper, xcolor]{article}
\usepackage{amsfonts}
\usepackage{amssymb,amsthm, amsmath,graphicx,bm,ebezier,amsthm}
\usepackage{hyperref}
\usepackage{url,lineno}
\usepackage[latin1]{inputenc}
\usepackage{arydshln}
\usepackage{algorithm}
\usepackage[noend]{algpseudocode}
\usepackage{lscape}

\newtheorem{theorem}{Theorem}
\newtheorem{definition}[theorem]{Definition}
\newtheorem{proposition}[theorem]{Proposition}
\newtheorem{corollary}[theorem]{Corollary}
\newtheorem{lemma}[theorem]{Lemma}
\newtheorem{conjecture}[theorem]{Conjecture}

\theoremstyle{remark}

\newtheorem{example}[theorem]{Example}

\def\e{\mathrm{e}}
\def\g{\mathrm{g}}
\def\n{\mathrm{n}}
\def\CaN{\mathcal{N}}
\def\CaH{\mathcal{H}}
\def\CaC{\mathcal{C}}
\def\S{\mathcal{S}}
\def\N{\mathbb{N}}

\def\Q{\mathbb{Q}}
\def\Fb{\mathrm{Fb}}

\title{An extension of Wilf's conjecture to affine semigroups}
\date{}
\author{
J. I. Garc\'{\i}a-Garc\'{\i}a
 \footnote{
     Departamento de Matem\'aticas,
     Universidad de C\'adiz, E-11510 Puerto Real  (C\'{a}diz, Spain).
     E-mail: ignacio.garcia@uca.es. Partially supported by
     MTM2014-55367-P and
     Junta de Andaluc\'{\i}a group FQM-366. }\\
D. Mar\'{\i}n-Arag\'on
 \footnote{
     E-mail: daniel.marinaragon@alum.uca.es.
     Partially supported by Junta de Andaluc\'{\i}a group FQM-366.}
     \\
 A. Vigneron-Tenorio
 \footnote{
 	Departamento de Matem\'aticas, Universidad de C\'adiz,
 	E-11406 Jerez de la Frontera (C\'{a}diz, Spain).
     E-mail: alberto.vigneron@uca.es.
     Partially supported by MTM2015-65764-C3-1-P (MINECO/FEDER, UE) and
     Junta de Andaluc\'{\i}a group FQM-366.}}

\begin{document}
\maketitle
\abstract{
	
Let $\CaC\subset \Q^p$ be a rational cone.
An affine semigroup $S\subset \CaC$ is a $\CaC$-semigroup whenever $(\CaC\setminus S)\cap \N^p$ has only a finite number of elements.

In this work, we study the tree of $\CaC$-semigroups,  give a method to generate it and  study their subsemigroups with minimal embedding dimension.
We extend  Wilf's conjecture for numerical semigroups to $\CaC$-semigroups and  give some families of $\CaC$-semigroups fulfilling the extended conjecture.
We also check that other conjectures on numerical semigroups seem to be also satisfied by $\CaC$-semigroups.

}



\section*{Introduction}
Let $\N$ denote the set of nonnegative integers.
An affine semigroup is a finitely generated submonoid $S$ of the additive monoid $(\N^p,+)$ with
$p$ a positive integer.
For an affine semigroup its associated rational cone $\CaC$ is the cone $\{\sum_{i=1}^n q_ix_i|n\in \N, q_i\in \Q_+, x_i\in S\}.$
It is well-known that any affine semigroup $S$ has a unique  minimal generating set whose cardinality is known as the embedding dimension of $S,$ and it is denoted by $\e(S).$

We introduce  the concept of $\CaC$-semigroup: given a rational cone $\CaC\subset\Q^p_+,$ an affine semigroup $S$ is called a $\CaC$-semigroup if the set $\CaH(S)=(\CaC\setminus S)\cap\N^p$ is finite.
For instance, the class of  $\N$-semigroups ($\CaC=\N$) is the set of numerical semigroups,
and $\N^p$-semigroups are called generalized numerical semigroups (see \cite{GenSemNp}).

Let $\prec$ be a monomial order satisfying that every monomial is preceded only by a finite number of monomials.
As explained in Section \ref{preliminaries}, every monomial order induces an order on $\N^p$ that we also denote by $\prec$.
Let $\S(\CaC )$ be the set of affine $\CaC$-semigroups.
For every $S\in\S(\CaC)$ the cardinality of $\CaH(S)$ is called the genus of $S$ and it is denoted by $\g(S).$
The maximum of $\CaH(S)$ with respect to $\prec $   is the Frobenius element of $S$, denoted by $\Fb(S)$. By convention,  $\Fb(\CaC)$ is the vector
$(-1,\ldots ,-1)\in \N^p$.
Denote by $\n(S)$ the cardinality of the finite set $\{x\in S \mid x\prec \Fb(S)\}$.
The Frobenius number of a $\CaC$-semigroup $S$ is defined as
$ \n(S)+\g(S)$ and denoted by $\CaN(\Fb(S))$. For numerical semigroups $\Fb(S)=\CaN(\Fb(S))$.

In 1978, Wilf proposed a conjecture related to the Diophantine Frobenius Problem (\cite{Wilf}) that  claims that the inequality
$\n(S)\cdot \e(S)\geq \Fb(S)+1$
is true for every numerical semigroup.
This conjecture still remains open, and it has become an important part of the Theory of Numerical Semigroup, for instance it has been  studied in \cite{Bras2008}, \cite{Fromentin}, \cite{Kaplan}, \cite{Moscariello}, \cite{libro_rosales} and the references therein. Most of these papers describe families of semigroups satisfying Wilf's conjecture.
In \cite{GenSemNp}, it is said it would be interesting to formulate potential extensions of Wilf's conjecture to the setting of $\N^p$-semigroups. One of the contributions of this work is to extend Wilf's conjecture formulating it in terms of $\CaC$-semigroups. Besides, we present several families of $\CaC$-semigroups satisfying this conjecture and the results of the computational tests applied to some $\CaC$-semigroups randomly obtained (see Table \ref{wilf_conj}).

A generating tree for an $\N^p$-semigroup is described in \cite{GenSemNp}, and in Definition \ref{tree} we generalize it for $\CaC$-semigroups. In order to compute this tree, we prove that the algorithm used in \cite{GenSemNp} for $\N^p$-semigroups can be used for $\CaC$-semigroups. Besides, we improve this algorithm making easier to obtain the minimal generating sets of the effective sons of a $\CaC$-semigroup.
Another contribution of this work is to characterize the minimal generating sets of the $\N^p$-semigroups with minimal embedding dimension.
Besides, we conjecture a lower bound of the embedding dimension of a $\CaC$-semigroup.

From the construction of the above trees, we obtain a data table (Table \ref{experiments}) with the amount of $\CaC$-semigroups with a fixed genus $g$, denoted by $n_g(\CaC)$.
The strong growth of $n_g(\CaC)$ makes very difficult a computational study
of $n_g(\CaC)$ in terms of $g$. Note that the easiest semigroup to study is $\N$, and for this semigroup,
it has only been possible to compute $n_g(\N)$ for $g\le 67$ (see \cite{Fromentin}).
These computational results are used to discuss the asymptotic behavior of $n_g(S).$

For this work,
the computations of trees of $\CaC$-semigroups have been done in a cluster of computers (\cite{super}), using python (\cite{python}) as programming language and the library {\tt mpich2-1.2.1} (\cite{mpi}) to parallelize the computations.
For all other computations, we used an Intel i7 with 32 Gb of RAM, and Mathematica (\cite{mathematica}).
We also have used the programs LattE (\cite{latte}) and Normaliz (\cite{normaliz}) for the computation of system of generators and for checking Wilf's conjecture.

The content of this work is organized as follows. Section \ref{preliminaries} sets some basic definitions and results related to affine semigroups.
Section \ref{sec_tree} provides primarily a method to construct the tree of $\CaC$-semigroups.
Section \ref{sec_min_dim_inm} studies the $\N^p$-semigroups with minimal embedding dimension and conjectures a lower bound for the embedding dimension of any $\CaC$-semigroup.
Section \ref{sec_wilf} introduces the extension of Wilf's conjecture  for $\CaC$-semigroups, gives some families of semigroups satisfying it and presents a computational study of this conjecture. The last section shows a data table with the number of $\CaC$-semigroups with a fixed genus and discusses the asymptotic behavior of this number.

\section{Preliminaries and notations}\label{preliminaries}

For any nonnegative integer $n,$ we denote by $[n]$ the set $\{1,\ldots ,n\}.$ The set $\{e_1,\ldots ,e_p\}$ denotes the standard basis of $\N^p.$

A monomial order is a total order on the set of all (monic) monomials in a given polynomial ring, satisfying the following two properties (see \cite{CoxLO}):
\begin{itemize}
\item if $u\preceq v$  and $w$ is any other monomial, then $uw\preceq vw$,
\item if $u$ is any monomial then $1\preceq u$.
\end{itemize}
If we translate these properties to $\N^p$ we obtain that a total order $\prec$ on $\N^p$ is monomial whenever:
\begin{itemize}
\item if $a \preceq b$ and $c\in \N^p$, then $a+c\preceq b+c$,
\item if $c\in\N^p$, then $0\preceq c$.
\end{itemize}
These conditions imply that if $a\preceq b$ and $c\preceq d$, then $a+c\preceq b+c$ and $b+c\preceq b+d$ which, by transitivity, implies $a+c\preceq b+d$. In particular, since $0\preceq c$, if $a\preceq b$, then $a=a+0\preceq  b+c.$ A monomial order can be expressed by a square matrix. For a nonsingular integer ($p\times p$)-matrix $M$ with rows $M_1,\ldots ,M_p,$ the $M$-ordering $\prec$ is defined by $a\prec b$ if and only if there exists an integer $i$ belonging to $[p-1],$ such that $M_1a=M_1b,\ldots ,M_ia=M_ib$ and $M_{i+1}a<M_{i+1}b.$ Every monomial order is equivalent to a matrix ordering.

In this work, for a given numerical semigroup with embedding dimension $2$, we use the following two well-known results (see \cite[Proposition 1.13]{libro_rosales}): if $T=\langle a,b\rangle$ is a numerical semigroup, then
\begin{itemize}
\item the Frobenius number of $T$ is $ab-a-b,$
\item the genus of $T$ is $\frac{ab-a-b+1}{2}.$
\end{itemize}
Another  result we use is that the minimal generating set of the numerical semigroup $[b+1,\infty)\cap \N$ is $\{b+1,\ldots,2b+1\}$ for every $b\in \N$ (see \cite[chapter 1, section 2]{libro_rosales}).

Let $S\subset \N^p$ be an affine semigroup, and $\tau$ be an extremal ray of the cone $\CaC$ associated to $S.$ If $S$ is a $\CaC$-semigroup, the embedding dimension of $S\cap \tau$ is one if and only if $S\cap \tau=\N^p\cap \tau.$ In particular, $S$ has at least two minimal generators in $\tau$ if and only if $S\cap \tau\neq\N^p\cap \tau.$

\section{Computing trees of $\CaC$-semigroups}\label{sec_tree}

Given a minimal generating set of a semigroup $S,$ the idea of the construction of the tree of semigroups is to remove a minimal generator in order to obtain a new semigroup $S'$ such that $S\setminus S'$ is the removed generator (\cite{GenSemNp}). The next lemma shows this construction and
a property which
can be used for improving the computation of the minimal system of generators of $S'.$

\begin{lemma}\label{construccion}
Let $S\subset \N^p$ be a semigroup minimally generated by the set $\{s_1,\ldots ,s_t\},$ $i\in[t]$, and $S'$ be the semigroup generated by $$\{s_1,\ldots ,s_{i-1},s_{i+1},\ldots ,s_t,2s_i,3s_i\}\cup \{s_i+s_j| j\in [t]\setminus \{i\}\}.$$ The following hold:
\begin{itemize}
\item $S'=S\setminus\{s_i\},$
\item the elements in $G=\{s_1,\ldots ,s_{i-1},s_{i+1},\ldots ,s_t\}$ are minimal generators of $S'.$
\end{itemize}
\end{lemma}

\begin{proof}
Trivially $S'=S\setminus\{s_i\}.$

Without loss of generality, we suppose that $s_t\in G$ is not a minimal generator of $S'$, so there exist $\lambda_1,\ldots ,\lambda_{t-1},\mu_1,\mu_2,\nu_1,\ldots ,\nu_{t-1}\in \N,$ such that
$$s_t=\sum_{j\in [t-1]} \lambda_js_j+2\mu_1s_t+3\mu_2s_t+\sum_{j\in [t-1]} \nu_j(s_t+s_j).$$
Hence, $\nu_j$ has to be zero for all $j,$ $\mu_1=\mu_2=0,$ and then $s_t=\sum_{j\in [t-1]} \lambda_js_j.$ In that case, $s_t$ is not a minimal generator of $S.$ We conclude that for any $s\in G,$ $s$ is a minimal generator of $S'.$
\end{proof}

A semigroup $S'$, obtained from $S$ by using the previous construction,  is called a descendant of $S.$

\begin{corollary}\label{coro_descendientes}
Let $S$ be a $\CaC$-semigroup with embedding dimension $t$. Then, $\e(S')\ge t-1$ for any  descendant $S'$  of $S.$
\end{corollary}

For every $S\in \S(\CaC)$, let $\{s_1 \prec \dots \prec s_t\}$ be the minimal system of generators of $S$ and let  $r$ be the minimum element such that $\Fb(S)\prec s_r\prec \dots \prec s_t$. The sets $S\setminus\{s_r\},\dots,S\setminus\{s_t\}$ are elements of $\S(\CaC)$, we call these semigroups the effective sons of $S$ and denote them by ${\frak F}(S).$
Note that the minimal generating sets of its effective sons can be computed in an easier way by using Lemma \ref{construccion}.
If $S\setminus \{s_i\}$ is an effective son of $S,$ the elements in the set $\{s_1,\ldots ,s_{i-1},s_{i+1},\ldots ,s_t\}$ are minimal generators of $S\setminus \{s_i\}.$ In order to obtain the other minimal generators of $S\setminus \{s_i\},$ we only have to get the minimal generators belonging to $\{2s_i,3s_i\}\cup \{s_i+s_j| j\in [t]\setminus \{i\}\}.$

\begin{definition}\label{tree}
The tree ${\cal  T}$ of $\CaC$-semigroups rooted in $\CaC$ is the tree with the set of vertices obtained recursively as follows:
\begin{enumerate}
	\item let $i=0$ and ${\tt L}_0=\{\CaC\},$
	\item $i=i+1,$
	\item define ${\tt L}_i=\cup_{ S\in {\tt L}_{i-1} } {\frak F}(S),$
	\item go to step 2.
\end{enumerate}
A pair of vertices $(S,S')$ is an edge if and only if
$S'$ is an effective son of $S$.
\end{definition}

The Frobenius element can be used to compute the tree of $\CaC$-semigroups rooted in a cone $\CaC$ in the same way as \cite{GenSemNp}, and this fact is proved in the following proposition.

\begin{proposition}
	The set ${\cal T}$ is a tree and its set of vertices is $\S(\CaC)$.
\end{proposition}
\begin{proof}
If $S'\in {\frak F}(S)$, then $S\setminus S'=\{ \Fb(S) \}$, and, therefore, if $S'\in {\frak F}(S_1)$ and $S'\in {\frak F}(S_2)$, then $S_1=S_2$.
Thus, ${\cal T}$ is a tree.

Let $S\in \S(\CaC)$. The set $S'=S\cup \{\Fb(S)\}$ is an element of $\S(\CaC)$ and $S\in {\frak F}(S')$. If repeat this process, after a finite number of steps, we obtain the set $\CaC$. Thus, $S$ is in the set of vertices of  ${\cal T}$.
\end{proof}

\section{Minimal embedding dimension of $\CaC$-semigroups}\label{sec_min_dim_inm}

We prove that the minimal embedding dimension of an  $\N^p$-semigroup is equal to  $2p$.
In general,
for $\CaC$-semigroups,
we propose a conjecture for a lower bound of the embedding dimension. Recall that a descendant semigroup has embedding dimension greater than or equal to ``its father's" embedding dimension minus 1 (Corollary \ref{coro_descendientes}).

An explicit description of the minimal generating set of an $\N^p$-semigroup with $\e(S)=2p$ is given in the following result.

\begin{proposition}\label{mult_2p}
Let $S$ be an $\N^p$-semigroup with embedding dimension $2p.$ Then, there exist $i\in [p],$
$\lambda_1,\lambda_2\in \N$ with $\gcd(\lambda_1,\lambda_2)=1,$ and $\{q_j\}_{j\in [p]\setminus \{i\}}\subset \N$
such that
$$\{e_1,\ldots ,e_{i-1},e_{i+1},\ldots ,e_p,\lambda_1 e_i,\lambda_2 e_i\}\cup \{e_i+q_je_j| j\in [p]\setminus \{i\}\}$$
is the minimal generating set of $S.$
\end{proposition}

\begin{proof}
Since $\e(S)=2p,$ there exists $j\in [p]$ such that $e_j\in S,$ otherwise there exist some integers
$\lambda_1^1,\ldots, \lambda_1^p, \lambda_2^1,\ldots, \lambda_2^p>1$ such that $S$ is minimally generated by $\{\lambda_1^1e_1,\ldots, \lambda_1^pe_p, \lambda_2^1e_1,\ldots, \lambda_2^pe_p\},$ but then $e_1+\nu e_2\in \N^p\setminus S$ for all $\nu \in \N$, and, therefore,  $S$ is not an $\N^p$-semigroup. So, we can assume that
$$G_1=\{e_1,\ldots, e_j,\lambda_1^{j+1}e_{j+1},\ldots,\lambda_1^pe_p,
\lambda_2^{j+1}e_{j+1},\ldots,\lambda_2^pe_p\}$$
is a subset of the minimal system of generators of $S$ for some $\lambda_i^j\in\N.$ Therefore, from the construction described in Lemma \ref{construccion}, there exist $q_1^{j+1},\ldots ,q_1^{p}, \ldots ,q_j^{j+1},\ldots ,q_j^{p} \in \N,$ such that the elements in $\cup_{_k=j+1}^p\{ e_k+q_1^{k}e_1,\ldots ,e_k+q_j^{k}e_j\}$ are in the set of minimal generators of $S.$ Thus, for $\e(S)=2p,$ $j$ has to be equal to $p-1,$ that is, only one canonical generator has been removed to construct $S.$ Assuming that this canonical generator is $e_i,$ $S$ is minimally generated by
$$\{e_1,\ldots ,e_{i-1},e_{i+1},\ldots ,e_p,\lambda_1 e_i,\lambda_2 e_i\}\cup \{e_i+q_je_j| j\in [p]\setminus \{i\}\}$$
with $\lambda_1,\lambda_2\in \N$ satisfying  $\gcd(\lambda_1,\lambda_2)=1,$ and $\{q_j\}_{j\in [p]\setminus \{i\}}\subset \N.$
\end{proof}

Two examples of $\N^3$-semigroups with minimal embedding dimension are:
$$\begin{array}{c}
\N^3\setminus \{e_1,e_1+e_3,e_1+2e_3,e_1+3e_3\}=\langle 2e_1,3e_1,e_2,e_3,e_1+2e_3,e_1+4e_2\rangle,\\
\N^3\setminus \{e_3,3e_3,e_3+e_2,e_3+e_1\}=\langle e_1,e_2,2e_3,5e_3,e_3+2e_2,e_3+2e_1\rangle.
\end{array}$$

The above result provides us a method for getting semigroups with fixed genus and minimal embedding dimension.

\begin{corollary}\label{easy_2p}
For any nonnegative integer $h$ and $i,k\in [p]$ with $i\neq k,$ the semigroup generated by
$$\{e_1,\ldots ,e_{i-1},e_{i+1},\ldots ,e_p,2e_i,3e_i,e_i+he_k\}\cup \{e_i+e_j| j\in [p]\setminus \{i,k\}\}$$
is an $\N^p$-semigroup with genus $h$ and embedding dimension $2p.$
\end{corollary}

In the following result, we give a lower bound of the embedding dimension of $\N^p$-semigroups.

\begin{theorem}
Let $p\in\N\setminus\{0\}$. If $S$ is an $\N^p$-semigroup with $S\neq \{0\}$ and $S\neq \N^p$, then $\e(S)\geq 2p$.
\end{theorem}

\begin{proof}
Note that any generalized numerical semigroup with genus 1 is minimally generated by a set of the form
$$\{e_1,\ldots ,e_{i-1},e_{i+1},\ldots ,e_p,2e_i,3e_i\}\cup \{e_i+e_j| j\in [p]\setminus \{i\}\}.$$
So, for genus $1$, the result holds.
For any semigroup $S$ with genus greater than or equal to 1, there exists $i\in [p]$ such that $e_i\notin S$. Thus, there exist elements in the minimal system of generators of $S$ of the form $q_1e_i,q_2e_i$ with $q_1,q_2\in\N$.

Assume that the result is fulfilled for a fixed genus $h\in\N,$ and let $S$ be an $\N^p$-semigroup with such genus. So, $\e(S)\ge 2p$, and in case $\e(S)> 2p,$ by Corollary \ref{coro_descendientes}, the descendants of $S$ have embedding dimension greater than or equal to $2p$.

Suppose now that $\e(S)=2p,$ and assume, without loss of generality, the removed canonical generator is $e_p.$ By Proposition \ref{mult_2p},
$$G=\{e_1,\ldots ,e_{p-1},\lambda _1 e_p,\lambda_2 e_p,e_p+q_1e_1,\ldots ,e_p+q_{p-1}e_{p-1}\}$$
is the minimal generating set of $S.$ Now, we apply the construction of Lemma \ref{construccion} to obtain the descendants of $S.$ If we remove a canonical generator $e_i$ belonging to $G,$ the minimal generating set of the corresponding descendant semigroup contains the minimal generators
$$\{e_1,\ldots e_{i-1},e_{i+1},\ldots,e_{p-1},2e_i,3e_i,\lambda _1 e_p,\lambda_2 e_p,e_p+q_1e_1,\ldots ,e_p+q_{p-1}e_{p-1}\};$$ thus, the embedding dimensions of the descendant semigroups obtained in this way are greater than or equal to $2p.$ Similarly, since the semigroup $S\cap \{x_1=\cdots =x_{p-1}=0\}$ does not contain the canonical generator $e_p,$ if we remove the minimal generator $\lambda_je_p,$ the embedding dimensions of its descendants are greater than or equal to $2p.$ To finish the proof, we remove from $G$ a generator of the form $e_p+q_{j}e_{j}.$ Assume that we remove  $e_p+q_{p-1}e_{p-1}.$ The corresponding descendant semigroup contains the minimal generators
\begin{multline}
\{e_1,\ldots e_{i-1},e_{i+1},\ldots,e_{p-1},\lambda _1 e_p,\\
\lambda_2 e_p,e_p+q_1e_1,\ldots ,e_p+q_{p-2}e_{p-2},e_p+(q_{p-1}+1)e_{p-1}\}.
\end{multline}
So, we conclude that the embedding dimensions of the $\N^p$-semigroups with genus $h+1$ are greater than or equal to $2p.$
\end{proof}

To finish this section, we propose a lower bound for the embedding dimension of $\CaC$-semigroups where the cone $\CaC$ is not necessarily the positive hyperoctant ($\Q^k_+$ with $k\in\N$).

\begin{conjecture}
Let $\CaC$ be an integer cone such that the dimension of the real vector space generated by $\CaC$ is $p.$ The embedding dimension of every $\CaC$-semigroup is greater than or equal to $2p$.
\end{conjecture}
An open question that arises from the above conjecture is when the bound $2p$ is reached.

\begin{example}
Let $\CaC$ be the rational cone with extremal rays $(3,1)$ and $(1,2).$
The integer cone $\CaC\cap \N^2$ has embedding dimension 4, and
the minimal embedding dimension of its sons  is showed in
Table \ref{tabla_minima_dimension_inmersion} up to genus $15$.
\begin{table}[h]
\centering
{\tiny\begin{tabular}{|c|c|c|c|c|c|c|c|c|c|c|c|c|c|c|c|c|c|c|c|}
  \hline
  genus & 0 & 1 & 2 & 3 & 4 & 5 & 6 & 7 & 8 & 9 & 10 & 11 & 12 & 13 & 14 & 15  \\ \hline
  min. $\e(\bullet)$ & 4 & 5 & 4 & 5 & 4 & 4 & 4 & 5 & 4 & 5 & 4 & 5 & 4 & 5 & 4 & 4\\
  \hline
\end{tabular}}
\caption{Computing minimal embedding dimension.}\label{tabla_minima_dimension_inmersion}
\end{table}
\end{example}

All the $\CaC$-semigroups computed to make Table \ref{wilf_conj} and Table \ref{experiments} satisfy the above conjecture.

\section{Extension of Wilf's conjecture}\label{sec_wilf}

Our goal in this section is to extend  Wilf's conjecture to $\CaC$-semigroups and  prove that this conjecture holds for several families of semigroups.
For this purpose, monomial orders satisfying that any monomial is preceded only by a finite number of other monomials are considered. A matrix ordering satisfies this property if all its entries in the first row of the matrix are positive. In the sequel, $\prec$ denotes a fixed weight order verifying the previous property, and we assume that $p\in \N$ is greater than or equal to two.

\begin{conjecture}
Let $S\subset \N^p$ be a $\CaC$-semigroup. The extended Wilf's conjecture is formulated as
\begin{equation}
\n(S)\cdot \e(S)\geq \CaN(\Fb(S))+1.
\end{equation}
\end{conjecture}

The above conjecture depends on the fixed monomial order, but according to our computational experiments, the conjecture seems true for every monomial order (Table \ref{wilf_conj}).
Below, we prove that it holds for some families of semigroups with any monomial order.

We start our study with a family of $\N^p$-semigroups with minimal embedding dimension.

\begin{lemma}
Let $h$ be an integer number greater than 1, $k\in [p-1],$ $\prec$ be a
monomial ordering in $\N^p$ as above,  and $S\subset \N^p$ be the semigroup minimally generated by
$$\{e_1,\ldots ,e_{p-1},2e_p,3e_p,e_p+he_k\}\cup \{e_p+e_i| i\in[p-1]\setminus \{k\}\}.$$
The $\N^p$-semigroup $S$ satisfies the extended Wilf's conjecture.
\end{lemma}

\begin{proof}
	
The set of gaps of $S$ is equal to $\{e_p,e_p+e_k,\dots,e_p+(h-1)e_k\}$, the genus is $h$, the embedding dimension is $2p$, and the Frobenius element is $e_p+(h-1)e_k$.

Since $\n(S) \e(S)= 2p \n(S)$  and  $\CaN(\Fb(S))+1=\n(S)+\g(S)=\n(S)+h$, the inequality
$ \n(S)\e(S)\geq \CaN(\Fb(S)) +1$ is equivalent to $\n(S) (2p-1)\geq h$.
We have that $0\prec e_k\prec \dots\prec (h-1)e_k\prec \Fb(S)$ and
$\{0,e_k,\dots,(h-1)e_k\}\subset S$. Therefore, $\n(S)\geq h$, so $\n(S) (2p-1)\geq h$ for every $p\geq 2$.
\end{proof}

We can prove the extended Wilf's conjecture for other families of $\CaC$-semigroups.

\begin{lemma}
Let $q$ be a nonzero nonnegative integer, $S\subset \N^p$ be the affine semigroup $\N^p\setminus \{e_j,\ldots ,(q-1)e_j\}$ with $j\in[p].$ Then, $S$ satisfies the extended Wilf's conjecture for every monomial order $\prec.$
\end{lemma}

\begin{proof}
Without loss of generality, assume that $j=1.$ Note that the semigroup $S$ is minimally generated by the set
$$
(\cup_{i=q}^{2q-1}\{ie_1\})
\cup
(\cup _{i\in [p]\setminus\{1\}}\{e_i\})
\cup
(\cup _{i\in [p]\setminus\{1\}}\{e_i+e_1,\ldots ,e_i+(q-1)e_1\}).
$$
So, $\e(S)= q+p-1+(p-1)(q-1)=pq$ and $\n(S)\e(S)=\n(S)pq.$
Since $\Fb(S)=(q-1)e_1$ for any monomial order $\prec,$  and $\g(S)=q-1,$ $\CaN(\Fb(S))+1=\n(S)+\g(S)=\n(S)+q-1.$
Thus,   $\n(S)\e(S)\geq \CaN(\Fb(S))+1$ if and only if  $\n(S)(pq-1)\geq q-1$, which is true.
\end{proof}

\begin{lemma}
Let $T\subset \N$ be the numerical semigroup minimally generated by $\{\lambda_1,\lambda_2\},$
$j\in [p]$, and $\{q_i\mid i \in [p]\setminus\{j\}\}\subset \N$. The affine semigroup $S=\N^p\setminus \{(x_1,\ldots ,x_p)\in\N^p| x_j\notin T\text{ and } x_i< q_i,\, \forall i\in [p]\setminus \{j\} \}$ satisfies the extended Wilf's conjecture for every monomial order $\prec.$
\end{lemma}

\begin{proof}
Assume that $j=1.$
Let $U$ be the set $\N\setminus T=[1, \lambda_1\lambda_2-\lambda_1-\lambda_2]\setminus T.$ So, the cardinality of $U$ is $\frac{\lambda_1\lambda_2-\lambda_1-\lambda_2+1}{2}.$
Note that $S$ is minimally generated by the set
$$
(\{\lambda_1e_1,\lambda_2e_1\})
\cup
(\cup_{i\in [p]\setminus\{1\}} \{e_i\})
\cup
(\cup _{\alpha\in U}\{\alpha+q_2e_2,\ldots ,\alpha+q_pe_p \})
.
$$
Therefore, $\e(S)=2+p-1+(p-1)\frac{\lambda_1\lambda_2-\lambda_1-\lambda_2+1}{2}=
p+1+(p-1)\frac{\lambda_1\lambda_2-\lambda_1-\lambda_2+1}{2}.$

The other elements involved in the extended Wilf's conjecture are now determined. The set $\N^p\setminus S$ is included in the hypercube
$ L=[0,\lambda_1\lambda_2-\lambda_1-\lambda_2]\times [0,q_2-1]\times \dots \times  [0,q_p-1]$.
Since every numerical semigroup generated by two elements is symmetric,
this set contains
as many points inside of $S$ as outside.
There are just $\frac{\lambda_1\lambda_2-\lambda_1-\lambda_2+1}{2}\prod_{i=2}^pq_i$ points in $\N^p\setminus S,$ and $\g(S)= \frac{\lambda_1\lambda_2-\lambda_1-\lambda_2+1}{2} \prod_{i=2}^pq_i.$
For every monomial order, the Frobenius element is $(\lambda_1\lambda_2-\lambda_1-\lambda_2)e_1+\sum_{i=2}^p(q_i-1)e_i$. Therefore,  $\CaN(\Fb(S))+1=\n(S)+\g(S)=\n(S)+\frac{\lambda_1\lambda_2-\lambda_1-\lambda_2+1}{2} \prod_{i=2}^pq_i,$ and $\alpha \prec \Fb(S)$ for all $\alpha \in L.$

So, for these semigroups the extended Wilf's conjecture can be formulated as:
$$\n(S)\left(p+(p-1)\frac{\lambda_1\lambda_2-\lambda_1-\lambda_2+1}{2} \right)\ge \frac{\lambda_1\lambda_2-\lambda_1-\lambda_2+1}{2} \prod_{i=2}^pq_i.
$$
Since $\n(S)$ is greater than $\frac{\lambda_1\lambda_2-\lambda_1-\lambda_2+1}{2} \prod_{i=2}^pq_i$ and $p+(p-1)\frac{\lambda_1\lambda_2-\lambda_1-\lambda_2+1}{2}\ge 1,$
the inequality holds.
\end{proof}

\begin{lemma}
Let $a$ and $b$ two nonnegative integers such that $a<b,$ and $\CaC\subset \N^2$ be the cone generated by $\{(1,0),(1,1)\}.$ The $\CaC$-semigroup $S =\langle (1,1)\rangle \cup (\CaC\setminus [0,b]\times [0,a])$ satisfies the extended Wilf's conjecture for every fixed monomial order $\prec.$
\end{lemma}

\begin{proof}
Note that $S$ is minimally generated by
$$
\{(1,1)\}
\cup
(\cup_{i=b+1}^{2b+1}\{(i,0)\})
\cup
(\cup _{i=1}^a\{(b+1,i)\})
\cup
(\cup _{i=a+2}^{b+1}\{(i,a+1)\}).
$$
Therefore, $\e(S)=2b+2,$ and $\g(S)=\frac 1 2 (1+a) (2b-a)$.
The Frobenius element $\Fb(S)$ is the integer point $(b,a)$ for the fixed order $\prec$, and
the extended Wilf's conjecture is equivalent to $$\n(S)(2b+1)\geq \frac 1 2 (1+a) (2b-a).$$
Since $\n(S)\geq a+1$, this inequality holds.
\end{proof}

Table \ref{wilf_conj} shows some samples of the computational test for checking the extended Wilf's conjecture for several $\CaC$-semigroups using different monomial orders.
To obtain this table we have proceeded  as follows:
\begin{itemize}
	\item For a fixed cone $\CaC$, we take randomly a subsemigroup $S_1$ of genus $1$. Again, we take randomly a subsemigroup $S_2\subset S_1$ with genus $2$. Repeating this process as many times as necessary,  we get a random subsemigroup $S_i$ of $\CaC$ of  genus $i$.
	\item Now, we define a monomial order taking a nonsingular matrix $M_i$ having all the elements of its first row greater than zero. The elements of this first row are random integer from $1$ to $10$, the rest of the elements of $M_i$ are random integer from $-10$ to $10$.
	\item We check if the extended Wilf's conjecture is satisfied by the semigroups $S_1,\dots, S_i$ using the orders defined by the matrices $M_i$.
\end{itemize}

With this procedure, we have verified, using random monomial orders, that the extended Wilf's conjecture is satisfied by the  elements of a random branch of the $\CaC$-tree appearing in Table \ref{wilf_conj}.
For computing Tables \ref{tabla_minima_dimension_inmersion} and \ref{wilf_conj}, we have used an Intel i7 with 32 Gb of RAM.

\setlength\dashlinedash{0.5pt}
\setlength\dashlinegap{1.5pt}
\setlength\arrayrulewidth{0.3pt}
\begin{table}[h]
\centering
{\small
\begin{tabular}{|c|c|c|c|}
  \hline
  Initial cone $\CaC/$ & $\e(\CaC)$ & genus & Ext. Wilf's conj.\\
  Extremal rays & & & is satisfied \\ \hline
  $\N^2$ & 2 & 1 to 500 & $\surd$ \\ \hdashline
  $\N^3$ & 3 & 1 to 500 & $\surd$ \\ \hdashline
  $\N^4$ & 4 & 1 to 500 & $\surd$ \\ \hdashline
  $\N^5$ & 5 & 1 to 500 & $\surd$ \\ \hdashline
$\{(13,1)\},(1,3)\}$
& 15 & 1 to 500 & $\surd$ \\ \hdashline
$\{(3,2,0),(0,1,0),(3,5,7),$ & & & \\
$(1,8,10),(13,21,33)\}$
& 62 & 1 to  500 & $\surd$ \\ \hdashline
$\{(5,0,1,2),(0,3,1,0),$ & & & \\
$(1,1,1,0),(0,2,1,1)\}$
& 11 & 1 to 500& $\surd$ \\ \hdashline
$\{(1,2,1,2,0),(1,0,0,0,1),(1,1,0,0,1),$ & & & \\
$(2,0,2,1,1),(1,1,1,1,3)\}$
& 12 & 1 to 500& $\surd$ \\

  \hline
\end{tabular}
}
\caption{Computational test of the extended Wilf's conjecture for $\CaC$-semigroups.}\label{wilf_conj}
\end{table}

\section{Some computational results on $\CaC$-semigroups}\label{sec_comp}

For a fixed cone $\CaC,$ denote by $n_g(\CaC)$ the number of the $\CaC$-semigroups with genus $g.$ In this section, a table with some computational results is presented. We have obtained it by parallel computing in a supercomputer (\cite{super}) using in most of the computations 320 cores.
The data obtained are in Table \ref{experiments}, and we use them to dicuss the asymptotic behavior of $n_g(\CaC).$

For numerical semigroups, there exist several open problems, conjectures and results about the asymptotic behavior of $n_g(\N).$ The first conjectures appear in \cite{Bras2008}:
\begin{itemize}
\item $n_g(\N)\geq n_{g-1}(\N)+n_{g-2}(\N),$ for $g\ge 2.$
\item $\lim_{g\to \infty }{\frac{n_{g-1}(\N)+n_{g-2}(\N)}{n_g(\N)}}=1.$
\item  $\lim_{g\to \infty }{\frac{n_{g}(\N)}{n_{g-1}(\N)}}=\varphi,$ where $\varphi$ is the golden ratio.
\end{itemize}
In \cite{Zhai}, the author proves that $\lim_{g \to \infty} n_g(\N) \varphi^{-g} = L,$
where $L$ is a constant.

For $\N^p$-semigroups, it is proved (see \cite{GenSemNp}) that the sequence $n_{g}(\N^p)^{(1)}$ (where $n_{g}(\N^p)^{(1)}$ is the cardinality of the set of $\N^p$-semigroups of genus $g$ whose set of gaps is supported on the union of the coordinate axes) is asymptotic to $\binom{g+p-1}{p-1} k^p\varphi ^g$ for some $k>0.$ If $n_{g}(\CaC)^{(1)}$ denotes the cardinality of the set of $\CaC$-semigroups of genus $g$ whose set of gaps is supported on the union of the extremal rays (assume $\CaC$ has $l$ extremal rays), it is easy to prove that $n_{g}(\CaC)^{(1)}$ is asymptotic to $\binom{g+l-1}{l-1} k^p\varphi ^g$ for some $k>0.$

The data collected in Table \ref{experiments} are insufficient to propose a new conjecture.
Anyway, and in view of the similarities of these tables with the tables obtained for numerical semigroups in \cite{Bras2008}, it seems that the inequalitis $n_g(\CaC)\geq n_{g-1}(\CaC)+n_{g-2}(\CaC)$
and $n_g(\CaC)\geq n_{g-1}(\CaC)$
are fulfilled  and that
the limit of the sequence $\left\{{\frac{n_{g}(\CaC)}{n_{g-1}(\CaC)}}\right\}_{g\ge 1}$ exists.

\begin{landscape}
\begin{table}[h]
\centering
{\tiny
\begin{tabular}{|c|c|c|c|c|c|c|c|c|c|c|c|c|}
  \hline
  \multicolumn{1}{|c} {\text{Cone/}} &  \multicolumn{3}{|c|}{$\N^2$} & \multicolumn{3}{|c|}{$\N^3$} & \multicolumn{3}{|c|}{$\langle (3,1),(1,2) \rangle_{\Q^+}$} & \multicolumn{3}{|c|}{$\langle (1, 1, 0), (0, 1, 0)$} \\
  \text{Extremal rays} & \multicolumn{3}{|c|}{} & \multicolumn{3}{|c|}{}  & \multicolumn{3}{|c|}{} & \multicolumn{3}{|c|}{$(1, 1, 1), (0, 1, 1), (3, 2, 1) \rangle_{\Q^+}$} \\
  \hdashline
  \text{genus}& $n_g$ & $\frac{n_g}{n_{g-1}}$ & $\frac{n_{g+1}+n_{g-1}}{n_{g}}$ & $n_g$ & $\frac{n_g}{n_{g-1}}$ & $\frac{n_{g+1}+n_{g-1}}{n_{g}}$ & $n_g$ & $\frac{n_g}{n_{g-1}}$ & $\frac{n_{g+1}+n_{g-1}}{n_{g}}$ & $n_g$ & $\frac{n_g}{n_{g-1}}$ & $\frac{n_{g+1}+n_{g-1}}{n_{g}}$ \\ \hline
  0 & 1 & & & 1 & & & 1 & & & 1 & & \\ \hdashline
  1 & 2 & 2 & & 3 & 3 & & 4 & 4 & & 5 & 5 &  \\ \hdashline

  2 & 7 & 3.5 & 0.428571 &
  15 & 5 & 0.266667 &
  17 & 4.25 & 0.294118 &
  32 & 6.4 & 0.1875
  \\ \hdashline

  3 & 23 & 3.28571 & 0.391304 &
  67 & 4.46667 & 0.268657 &
  63 & 3.70588 & 0.333333 &
  179 & 5.59375 & 0.206704
  \\ \hdashline

  4 & 71 & 3.08696 & 0.422535 &
  292 & 4.35821 & 0.280822 &
  236 & 3.74603 & 0.338983 &
  960 & 5.36313 & 0.219792
  \\ \hdashline

  5 & 210 & 2.95775 & 0.447619 &
  1215 & 4.16096 & 0.295473 &
  838 & 3.55085 & 0.356802 &
  4951 & 5.15729 & 0.230055
  \\ \hdashline

  6 & 638 & 3.0381 & 0.440439 &
  5075 & 4.17695 & 0.296946 &
  2896 & 3.45585 & 0.370856 &
  25049 & 5.05938 & 0.235977
  \\ \hdashline

  7 & 1894 & 2.96865 & 0.44773 &
  20936 & 4.12532 & 0.300439 &
  9764 & 3.37155 & 0.382425 &
  124395 & 4.96607 & 0.241167
  \\ \hdashline

  8 & 5570 & 2.94087 & 0.454578 &
  85842 & 4.10021 & 0.30301 &
  32381 & 3.31637 & 0.39097 &
  608825 & 4.89429 & 0.245463
  \\ \hdashline

  9 & 16220 & 2.91203 & 0.460173 &
  349731 & 4.07412 & 0.305315 &
  106060 & 3.27538 & 0.397369 &
  2943471 & 4.83467 & 0.2491
  \\ \hdashline

  10 & 46898 & 2.89137 & 0.464625 &
  1418323 & 4.05547 & 0.307104 &
  343750 & 3.24109 & 0.402737 &
  14084793 & 4.7851 & 0.252208
  \\ \hdashline

  11 & 134856 & 2.87552 & 0.46804 &
  5731710 & 4.04119 & 0.308469 &
  1103235 & 3.20941 & 0.407719 &
  66814010 & 4.7437 & 0.254861
  \\ \hdashline

  12 & 386354 & 2.86494 & 2.86494 &
  23100916 & 4.03037 & 0.309513 &
  3509368 & 3.18098 & 0.412321 &
  ... & ... & ... \\ \hdashline

  13 & 1102980 & 2.85484 & 0.472547 &
  92882954 & 4.02075 & 0.310419 &
  11075932 & 3.1561 & 0.416453 &
  ... & ... & ... \\ \hdashline

  14 & 3137592 & 2.84465 & 0.474674 &
  ... & ... & ... &
  34719935 & 3.13472 & 0.420084 &
  ... & ... & ... \\ \hdashline

  15 & 8892740 & 2.83426 & 0.476858 &
  ... & ... & ... &
  108185393 & 3.11594 & 0.423309 &
  ... & ... & ... \\ \hdashline

  16 & 25114649 & 2.82417 & 0.479017 &
  ... & ... & ... & ... & ... & ... & ... & ... & ...\\ \hdashline
  17 & 70686370 & 2.81455 & 0.481102 &
   ... & ... & ... & ... & ... & ... & ... & ... & ...\\ \hdashline
  18 & 196981655 & 2.7867 & 0.486345 &
   ... & ... & ... & ... & ... & ... & ... & ... & ...\\
  \hline
\end{tabular}
}
\caption{Computational experiments on $n_g(\CaC).$}\label{experiments}
\end{table}
\end{landscape}

\noindent
{\bf Acknowledgements.} 
The authors would like to thank Shalom Eliahou for his helpful comments and suggestions related to this work.

\end{document}